\documentclass[a4paper,twoside]{amsart}
\usepackage{amsmath}
\usepackage{amsthm,color}
\usepackage{amssymb}
\usepackage{amsmath}
\usepackage{amssymb}
\usepackage{ifpdf}
\ifpdf
  \usepackage[colorlinks,linkcolor=red,anchorcolor=blue,citecolor=green]{hyperref}
\else
\fi

\newtheorem{theorem}{Theorem}[section]

\newtheorem{corollary}[theorem]{Corollary}

\newtheorem{lemma}[theorem]{Lemma}

\newtheorem{question}[theorem]{Question}

\numberwithin{equation}{section}

\newcommand{\PA}{\operatorname{PA}}
\newcommand{\SSy}{\operatorname{SSy}}

\newcommand{\ZF}{\operatorname{ZF}}
\newcommand{\ZFC}{\operatorname{ZFC}}

\newcommand{\MA}{\operatorname{MA}}

\begin{document}

\title[On Non-standard Models of Arithmetic]{On Non-standard Models of Arithmetic\\
with Uncountable Standard Systems}


\author{Wei Wang}
\address{Institute of Logic and Cognition and Department of Philosophy\\Sun Yat-Sen University\\Guangzhou, China}
\email{wwang.cn@gmail.com}

\thanks{The author was partially supported by China NSF Grant 11971501.
The results here have been presented in several occasions.
The author thanks various logicians for their helpful opinions,
 in particular, Tin Lok Wong, Jiacheng Yuan, Yinhe Peng and Victoria Gitman.}

\begin{abstract}
In 1960s, Dana Scott gave a recursion theoretic characterization of 
 standard systems of countable non-standard models of arithmetic,
 i.e., collections of sets of standard natural numbers coded in non-standard models.
Later, Knight and Nadel proved that Scott's characterization also applies to 
 non-standard models of arithmetic with cardinality $\aleph_1$.
But the question, 
 whether the limit on cardinality can be removed from the above characterization, 
 remains a long standing question,
 known as the Scott Set Problem.
This article presents two constructions of non-standard models of arithmetic 
 with non-trivial uncountable standard systems.
The first one leads to a new proof of the above theorem of Knight and Nadel,
 and the second proves the existence of models 
 with non-trivial standard systems of cardinality the continuum.
A partial answer to the Scott Set Problem under certain set theoretic hypothesis 
 also follows from the second construction.
\end{abstract}

\maketitle

\section{Introduction}\label{s:introduction}

Given a non-standard model of arithmetic $M$, 
 i.e. a model of arithmetic different from $\mathbb{N}$,
 a subset of $\mathbb{N}$ is \emph{coded} in $M$ iff 
 it equals to the intersection of $\mathbb{N}$ and some definable subset of $M$.
The \emph{standard system} of $M$, denoted by $\SSy(M)$, 
 is the collection of subsets of $\mathbb{N}$ that are coded in $M$,
 and has proved important in the theory of models of arithmetic.
As an example, we recall a theorem of Friedman and also some related concepts.
 
Given a model $M$ and a finite set of parameters $\vec{a} = (a_1,\ldots,a_n)$ from $M$,
 a type $p$ of $M$ over $\vec{a}$ is \emph{recursive},
 iff $p$ is in a fixed finite set of free variables $\vec{x}$ and the following set of formulas is recursive
$$
  \{\varphi(\vec{x},\vec{y}): \varphi(\vec{x},\vec{a}) \in p\}.
$$
A model $M$ is \emph{recursively saturated}
 iff every recursive type of $M$ is realized in $M$.
Every infinite model can be elementarily extended to a recursively saturated model.
For countable recursively saturated models of $\PA$,
 standard systems are their own blueprints kept by themselves,
 according to Friedman's Embedding Theorem below
 (which can also be found in more recent literature like \cite{Kossak.Schmerl:2006.book}).
 
\begin{theorem}[Friedman \cite{Friedman:1971}]\label{thm:Friedman}
Suppose that $M$ and $N$ are countable recursively saturated models of $\PA$,
 and they are elementarily equivalent.
Then $\SSy(M) \subseteq \SSy(N)$ iff
 there exists an elementary embedding sending $M$ to an initial segment of $N$.
Moreover, $M$ and $N$ are isomorphic iff $\SSy(M) = \SSy(N)$.
\end{theorem}

In 1962, Scott \cite{Scott:1962} proved that the standard system $\mathcal{S}$ of a non-standard model (of arithmetic)
 always satisfies some recursion theoretic conditions below.
\begin{enumerate}
 \item[(S1)] If $X$ and $Y$ are both in $\mathcal{S}$ then so is $X \oplus Y = 2X \cup (2Y + 1)$.
 \item[(S2)] If $X \in \mathcal{S}$ and $Y$ is recursive in $X$ then $Y \in \mathcal{S}$.
 \item[(S3)] If $\mathcal{S}$ contains an infinite binary tree $T$ then $\mathcal{S}$ also contains an infinite path of $T$.
\end{enumerate}
Today, a collection of subsets of $\mathbb{N}$ satisfying (S1-3) above is called a \emph{Scott set}.
Scott also proved the reverse direction for countable Scott sets.
 
\begin{theorem}[Scott \cite{Scott:1962}]\label{thm:Scott}
A countable $S$ is a Scott set iff
 $S = \SSy(M)$ for some countable non-standard model $M$ of $\PA$.
\end{theorem}

Knight and Nadel \cite{Knight.Nadel:1982.jsl.models} extended Scott's Theorem to some uncountable Scott sets.

\begin{theorem}[Knight and Nadel \cite{Knight.Nadel:1982.jsl.models}]\label{thm:KN}
Every Scott set of cardinality $\leq \aleph_1$ is the standard system of a non-standard model of $\PA$.
\end{theorem}

But the question, whether Scott's Theorem holds for arbitrary Scott sets, remains open,
and has been named the \emph{Scott Set Problem} in literature (e.g., see \cite{Kossak.Schmerl:2006.book}).

\begin{question}[Scott Set Problem]\label{q:Scott-set}
Does every Scott set equal to $\SSy(M)$ of some non-standard model of $\PA$?
\end{question}

Nevertheless, there are some interesting partial answers to the Scott Set problem.
For example, Gitman \cite{Gitman:2008.Scott} proved that
 certain uncountable Scott sets could equal to standard systems of non-standard models,
 under the Proper Forcing Axiom.
People also investigate parallel questions in other first order theories,
 e.g., real closed fields and Presburger arithmetic in \cite{Dolich.Knight.ea:2015.aml}.

This article presents some attempts to understand the Scott Set Problem.

In \S \ref{s:Ehrenfeucht}, we shall see an alternative proof of Theorem \ref{thm:KN} of Knight and Nadel.
Indeed, there have been several alternative proofs of Theorem \ref{thm:KN}.
It may be interesting to note that 
 all known proofs of Theorem \ref{thm:KN} rely on recursively saturated models,
 including the original proof and a recent one published in \cite{Dolich.Knight.ea:2015.aml}.
These may be read as evidences reinforcing the tie between standard systems and recursively saturated models.
However, the proof presented here does not need recursive saturation
 and looks more straightforward.
 
In \S \ref{s:uncountable-SSy}, we shall prove that there do exist models of $\PA$
 which have non-trivial standard systems of cardinality the continuum.
From the proof of this existence, we shall be able to derive some partial answer
 to the Scott Set Problem.
 
We finish this section by recalling some notations and basic knowledge
 which will be used in the rest of the article.
 
Above we use $\mathbb{N}$ to denote the standard model of arithmetic.
We shall also use $\mathbb{Q}$ to denote the set of standard rational numbers.
But in many cases, it is more convenient to use $\omega$ for $\mathbb{N}$,
 as in set theory.
A subset of $\omega$ is identified with its characteristic function.
Given a non-standard $M \models \PA$, 
 every $a \in M$ can be regarded as a binary sequence,
 with the $i$-th bit denoted by $(a)_i$.
If $a \in M - \omega$, $a$ \emph{codes} the following subset of $\omega$
$$
  \{i \in \omega: M \models (a)_i = 1\}.
$$
The \emph{standard system} of $M$, denoted by $\SSy(M)$,
 is the collection of subsets of $\omega$ coded by some $a \in M - \omega$.
It is easy to see that the elements of $\SSy(M)$ coincide with
 subsets of $\omega$ which are intersections of $\omega$ and definable subsets of $M$.

Since $\PA$ admits definable Skolem functions,
 we can build an elementary extension of a given $M \models \PA$,
 by building a type $p(x)$ of $M$,
 and then take an extension (called a \emph{$p(x)$-extension} of $M$) with its universe consisting of $F(b)$,
 where $b$ is a fixed realization of $p(x)$ and $F$ ranges over all unary functions definable in $M$.
If $p(x)$ is bounded, i.e., $p(x) \vdash x < a$ for some $a \in M$,
 we may even assume that $F$ is defined on $\{i \in M: i < a\}$ and so is (coded by) an element of $M$.

A collection $\mathcal{S}$ of subsets of $\omega$ satisfying (S1,S2) above
 is called a \emph{Turing ideal}.
So Scott sets are Turing ideals satisfying (S3).
Suppose that $\mathcal{I}$ is a Turing ideal.
A set is \emph{$\mathcal{I}$-recursive} iff it is recursive in some set in $\mathcal{I}$.
Given another $X \subseteq \omega$,
 let $\mathcal{I} \oplus X$ denote the following collection
$$
  \{Z \subseteq \omega: Z \text{ is recursive in } \mathcal{I} \oplus X\},
$$
which is clearly also a Turing ideal.

For a better background in models of arithmetic, we recommend \cite{Kossak.Schmerl:2006.book}.

\section{A Straightforward Construction}\label{s:Ehrenfeucht}

This sections presents an alternative proof of Theorem \ref{thm:KN} of Knight and Nadel,
 via the following result of Ehrenfeucht.
It is clear that Theorem \ref{thm:Ehrenfeucht} implies Theorem \ref{thm:KN}.
Theorem \ref{thm:Ehrenfeucht} itself is known provable via Friedman's Embedding Theorem \ref{thm:Friedman} (see \cite{Gitman:2008.Scott}).
So the known proof relies on recursive saturation.
Below we present a proof of Theorem \ref{thm:Ehrenfeucht} via a straightforward construction,
 which does not need recursive saturation.

\begin{theorem}[Ehrenfeucht]\label{thm:Ehrenfeucht}
Let $\mathcal{S}$ be a Scott set and $M$ a countable non-standard model of $\PA$ with $\SSy(M) \subseteq \mathcal{S}$. 
For every $X \in \mathcal{S}$ there exists a countable elementary extension $N$ of $M$ with $X \in \SSy(N) \subseteq \mathcal{S}$.
\end{theorem}

Let $\mathcal{S}, M$ and $X$ be as in the statement of the above theorem.
Fix $a \in M - \omega$.
We shall construct a type $p(x)$ of $M$ s.t. $p \vdash x < 2^a$ and
 then let $N$ be a $p(x)$-extension of $M$.
As $M$ is countable, $N$ will be countable as well.
The type $p(x)$ will be constructed as a union of types $(p_i(x): i \in \omega)$.

As $M$ is countable, we can fix a list $(f_i: i \in \omega)$ of all $f \in M$ 
 which maps $2^a = \{n \in M: n < 2^a\}$ to $M$.
Assume that $f_0$ is the identify function on $2^a$.

Let
\begin{align*}
  p_0(x) &= \{x < 2^a\} \cup \{(x)_n = X(n): n \in \omega\} \\
    &= \{x < 2^a\} \cup \{(f_0(x))_n = X(n): n \in \omega\}. 
\end{align*}
As $a > \omega$, $p_0(x)$ is finitely realizable in $M$.
Also note that $p_0(x)$ is recursive in $X$,
and if $N$ is a $p_0(x)$-extension of $M$ then $X \in \SSy(N)$.

Suppose that for $k \in \omega$ we have the following data
\begin{itemize}
 \item $X_0,\ldots,X_k \subseteq \omega$ s.t. 
  $X_0 = X$ and each $X_i$ is in $\mathcal{S}$;
 \item A type of $M$ as below
 $$
  p_k(x) = \{x < 2^a\} \cup \{(f_i(x))_n = X_i(n): i \leq k, n \in \omega\}.
 $$
\end{itemize}
Note that $p_k(x)$ is recursive in $\bigoplus_{i \leq k} X$
 and thus recursive in $\mathcal{S}$,
 and that if $N$ is a $p_k(x)$-extension of $M$ and $b \in N$ realizes $p_k(x)$
  then $f_i(b)$ codes $X_i$ for all $i \leq k$.

Let $T$ be the set of tuples $\vec{\sigma} = (\sigma_i: i \leq k+1)$ s.t.
 $\sigma_i$'s are finite binary sequences of equal length and
 in $M$ the following set is not empty
$$
  W(\vec{\sigma}) = \{c < 2^a: \forall i \leq k+1, n < |\sigma_i| ((f_i(c))_n = \sigma_i(n))\}.
$$
So $T$ is in $\SSy(M)$.

Fix $m \in \omega$.
For each $i \leq k$,
 let $\sigma_i$ be the initial segment of $X_i$ of length $m$.
As $p_k(x)$ is finitely realizable in $M$,
 there exists $c \in M$ s.t. $c < 2^a$ and $(f_i(c))_n = \sigma_i(n)$ for each $i \leq k$ and $n < m$.
Define a binary sequence $\sigma_{k+1}$ of length $m$ by letting $\sigma_{k+1}(n) = (f_i(c))_n$ for $n < m$.
Then for this tuple $\vec{\sigma} = (\sigma_i: i \leq k+1)$, 
 the set $W(\vec{\sigma})$ contains $c$ and thus is not empty.
So $\vec{\sigma} \in T$.
This shows that $T$ is infinite.

Let $T'$ be the set of $\tau \in 2^{<\omega}$
 s.t. if $\tau_i$ is the initial segment of $X_i$ of length $|\tau|$
  then $(\tau_0,\ldots,\tau_{k},\tau) \in T$.
By the above paragraph, 
 $T'$ is an infinite binary tree recursive in $\bigoplus_{i \leq k} X_i \oplus T$
 and thus in $\mathcal{S}$.
So by (S3) in the definition of Scott set,
 $\mathcal{S}$ contains an infinite path of $T'$, denoted by $X_{k+1}$.

Hence the following set is a type of $M$,
$$
  p_{k+1}(x) = p_k(x) \cup \{(f_{k+1}(x))_n = X_{k+1}(n): \forall n \in \omega\},
$$
and $p_{k+1}(x)$ is recursive in $\mathcal{S}$.

Finally, let $p(x) = \bigcup_k p_k(x)$.
Then $p(x)$ is a type of $M$,
 and if $b$ realizes $p(x)$
  then $b$ codes $X$ and each $f_i(b)$ codes $X_i$ which is in $\mathcal{S}$.
So any $p(x)$-extension of $M$ is a desired model $N$.

This ends the proof of Ehrenfeucht's Theorem \ref{thm:Ehrenfeucht}.

\section{Uncountable Standard Systems}\label{s:uncountable-SSy}

Here we shall prove the existence of non-standard models
 whose standard systems are non-trivial and have cardinality the continuum.

\begin{theorem}[$\ZF$]\label{thm:nontrivial-continuum-SSy}
For every non-standard countable $N \models \PA$,
 there are $(M_\mathcal{X}: \mathcal{X} \subseteq 2^\omega)$ s.t.
  each $M_\mathcal{X}$ is an elementary extension of $N$, $|M_\mathcal{X}| = |\SSy(M_\mathcal{X})| = \max\{\omega,|\mathcal{X}|\}$ and
$$
  \mathcal{X} \subseteq \mathcal{Y} \Leftrightarrow M_\mathcal{X} \preceq M_\mathcal{Y} \Leftrightarrow \SSy(M_\mathcal{X}) \subseteq \SSy(M_\mathcal{Y}).
$$
   
Moreover, if $\mathcal{A} \subset 2^\omega - \SSy(N)$ is countable
 then we can have $\mathcal{A} \cap \SSy(M_\mathcal{X}) = \emptyset$ for all $\mathcal{X} \subseteq 2^\omega$.
\end{theorem}

Fix $a \in N - \omega$.
 For types, we shall mean types of $N$.

For each $n \leq \omega$ and $\sigma \in 2^n$,
 let $x_\sigma$ be a variable.
 If $m \leq n$, $\sigma_1,\ldots,\sigma_k \in 2^n$ and 
  $\phi(x_{\sigma_1},\ldots,x_{\sigma_k})$ contains \emph{no} quantifiers over any $x_{\sigma_i}$,
  then the \emph{$m$-reduct} of $\phi$ is the formula
  $$
    \phi(x_{\sigma_1},\ldots,x_{\sigma_k}; x_{\sigma_1 \upharpoonright m}, \ldots, x_{\sigma_k \upharpoonright m}),
  $$
  i.e., the formula obtained by simultaneously substituting $x_{\sigma_i \upharpoonright m}$'s for $x_{\sigma_i}$'s in $\phi$,
  where $\sigma \upharpoonright m$ is the sequence consisting of the first $m$ bits of $\sigma$.
  We also call the original $\phi$ an \emph{$n$-ramification} of its $m$-reduct.
  
A \emph{condition} $p$ is a finite type in $(x_\sigma: \sigma \in 2^{n_p})$ for some $n_p \in \omega$,
 s.t. $p$ contains no quantifiers over any $x_\sigma$, 
   $p \vdash x_\sigma \in 2^a$ and 
   there exists a positive $r \in \mathbb{Q}$ with
    $$
    	N \models |p(N)| > r (2^{a})^{2^{n_p}},
    $$
    where $p(N)$ is the set of realizations of $p$ in $N$.
 Let $P$ be the set of conditions.
 For $p,q \in P$, $q \leq p$ iff
 $n_q \geq n_p$ and $q$ contains every $n_q$-ramification of every $\phi \in p$.
 
For a descending sequence $\vec{p} = (p_i: i \in \omega)$ from $P$ s.t. $\lim_i n_{p_i} = \infty$,
 let $G_{\vec{p}}$ be the set of $\phi(x_{f_1}, \ldots x_{f_k})$ ($f_i \in 2^\omega$) s.t.
  the $n_{p_i}$-reduct of $\phi$ is in $p_i$ for some $i$.

\begin{lemma}\label{lem:nontrivial-SSy-generic-type}
If $\vec{p}$ and $G_{\vec{p}}$ are as above then $G_{\vec{p}}$ is a type of $N$.
\end{lemma}

\begin{proof}
For every finite subset $H$ of $G_{\vec{p}}$,
 there is a fixed $i$ s.t. formulas in $H$ are ramifications of formulas in $p_i$.
As $p_i$ is a finite type of $N$,
 $p_i$ is realized in $N$ by some tuple,
 which also realizes $H$.
\end{proof}

To construct $G_{\vec{p}}$ as above,
 we should be able to extend conditions non-trivially.
 
\begin{lemma}\label{lem:nontrivial-SSy-extension}
Each condition $p$ can be extended to another condition $q$ with $n_q > n_p$.
\end{lemma}

\begin{proof}
Let $q$ be the set of $(n_p+1)$-ramifications of all formulas in $p$.
Then $q$ is as desired.
\end{proof}

The lemma below will be used to that
 if $\mathcal{X}$ and $\mathcal{Y}$ are different subsets of $2^\omega$ then $\SSy(M_\mathcal{X})$ and $\SSy(M_\mathcal{Y})$ are different.

\begin{lemma}\label{lem:nontrivial-SSy-incomparability}
Suppose that $p \in P$ and $F: N^{k} \to N$ is definable in $N$.
 Then there exists $q \leq p$ s.t.
  $n_q = n_p$ and
   every $(\sigma,\sigma_1,\ldots,\sigma_k)$ from $2^{n_q}$ with $\sigma \neq \sigma_1, \ldots, \sigma_k$ 
   corresponds to some $i < \omega$ with
    $q \vdash (x_\sigma)_i \neq (F(x_{\sigma_1}, \ldots, x_{\sigma_k}))_i$.
\end{lemma}
 
\begin{proof}
It suffices to prove that
   every $(\sigma,\sigma_1,\ldots,\sigma_k)$ from $2^{n_p}$ with $\sigma \neq \sigma_1, \ldots, \sigma_k$ 
	corresponds to some $i < \omega$ and $q \leq p$ s.t.
	$n_q = n_p$ and $q \vdash (x_\sigma)_i \neq (F(x_{\sigma_1}, \ldots, x_{\sigma_k}))_i$.
	
Let $n = n_p$, 
 $r \in \mathbb{Q}$ be positive s.t. $|p(N)| > r (2^{a})^{2^n}$ in $N$.
Fix $(\sigma,\sigma_1,\ldots,\sigma_k)$ from $2^{n}$ as above.
For each $m \in \omega$, in $N$ the cardinality of the following set
 $$
  \{(b_\tau: \tau \in 2^n) \in p(N): \forall i < m ((b_\sigma)_i = (F(b_{\sigma_1}, \ldots, b_{\sigma_k}))_i)\}
 $$
 is at most $2^{-m} (2^{a})^{2^n}$,
 since each $(b_\tau: \tau \in 2^n)$ in the set has the first $m$ bits of $b_\sigma$ determined by $(b_{\sigma_1},\ldots,b_{\sigma_k})$.
 Hence, there must be some $i \in \omega$ and some positive $\epsilon \in \mathbb{Q}$ s.t.
  in $N$,
 $$
  |\{(b_\tau: \tau \in 2^n) \in p(N): (b_\sigma)_i \neq (F(b_{\sigma_1}, \ldots, b_{\sigma_k}))_i\}| > \epsilon (2^{a})^{2^n}.
 $$
 So $q = p \cup \{(x_\sigma)_i \neq (F(x_{\sigma_1}, \ldots, x_{\sigma_k}))_i\}$ is as desired.
\end{proof}

To exclude certain $g$'s from $\SSy(M_\mathcal{X})$'s,
 we prove one more lemma below.

\begin{lemma}\label{lem:nontrivial-SSy-avoidance}
If $g \not\in \SSy(N)$, $F: N^k \to N$ is definable in $N$ and $p \in P$
 then there exists $q \leq p$ s.t.
  $n_q = n_p$ and 
  every $(\sigma_1,\ldots,\sigma_k) \in (2^{n_q})^k$ corresponds to some $i < \omega$
   with $q \vdash g(i) \neq (F(x_{\sigma_1}, \ldots, x_{\sigma_k}))_i$.
\end{lemma}

\begin{proof}
It suffices to prove that
 every $(\sigma_1,\ldots,\sigma_k) \in (2^{n_q})^k$ corresponds to some $q \leq p$ and $i < \omega$ s.t.
  $n_q = n_p$ and $q \vdash g(i) \neq (F(x_{\sigma_1},\ldots,x_{\sigma_k}))_i$.
  
Fix all the data and $(\sigma_1,\ldots,\sigma_k) \in (2^{n_q})^k$ as above.
  Define a function $h: \omega \to 2$ as follows.
   Let $h(i)$ be the least $j < 2$, s.t. in $N$,
   $$
    |\{(b_\tau: \tau \in 2^{n_p}) \in p(N): j = (F(b_{\sigma_1}, \ldots, b_{\sigma_k}))_i\}| \geq |p(N)|/2.
   $$
  So $h \in \SSy(N)$.
  Since $g \not\in \SSy(N)$, we can pick $i < \omega$ s.t. $g(i) \neq h(i)$.
  Then $q = p \cup \{g(i) \neq (F(x_{\sigma_1}, \ldots, x_{\sigma_k}))_i\}$ is as desired.
\end{proof}

By the above lemmata, we can construct $\vec{p} = (p_i: i \in \omega)$ s.t.
\begin{enumerate}
 \item $p_{i+1} \leq p_i \in P$;
 \item $\lim_i n_{p_i} = \infty$;
 \item For each $p_i$ and each $N$-definable function $F: N^k \to N$,
  there exist $p_j \leq p_i$ and $m < \omega$, s.t.
  if $\sigma_1,\ldots,\sigma_k \in 2^{n_{p_j}}$ and $\sigma \in 2^{n_{p_j}} - \{\sigma_1,\ldots,\sigma_k\}$
   then $p_j \vdash \exists n < m ((x_\sigma)_n \neq (F(x_{\sigma_1},\ldots,x_{\sigma_k}))_n)$;
 \item For each $p_i$, each $g \in \mathcal{A}$ and each $N$-definable function $F: N^k \to N$,
  there exist $p_j \leq p_i$ and $m < \omega$ s.t.
  every $\sigma_1,\ldots,\sigma_k \in (2^{n_{p_j}})^k$ corresponds to some $n < m$ with
   $p_j \vdash g(n) \neq (F(x_{\sigma_1}, \ldots, x_{\sigma_k}))_n$.
\end{enumerate}
So $G_{\vec{p}}$ is a type in $(x_f: f \in 2^\omega)$ over $N$.
 Let $(a_f: f \in 2^\omega)$ be a realization of $G_{\vec{p}}$ in some $N' \succ N$.
 If $\mathcal{X} \subseteq 2^\omega$, let $M_\mathcal{X}$ be the Skolem hull of $N \cup \{a_f: f \in \mathcal{X}\}$ in $N'$.
 Then $M_\mathcal{X}$'s ($\mathcal{X} \subseteq 2^\omega)$ are as desired.

This finishes the proof of Theorem \ref{thm:nontrivial-continuum-SSy}.
 
\begin{corollary}[$\ZFC + \MA$]
\label{cor:MA-nontrivial-continuum-SSy}
For every non-standard $N \models \PA$ s.t. $|N| < 2^\omega$, 
 there exists a family $(M_\mathcal{X}: \mathcal{X} \subseteq 2^\omega)$ s.t.
  $N \prec M_\mathcal{X}$, $|M_\mathcal{X}| = |\SSy(M_\mathcal{X})| = \max\{|\SSy(N)|,|\mathcal{X}|\}$ and
$$
  \mathcal{X} \subseteq \mathcal{Y} \Leftrightarrow M_\mathcal{X} \preceq M_\mathcal{Y} 
  \Leftrightarrow \SSy(M_\mathcal{X}) \subseteq \SSy(M_\mathcal{Y}).
$$
  
Moreover, if $\mathcal{A} \subset 2^\omega - \SSy(N)$ has cardinality $< 2^\omega$
 then we can have $\mathcal{A} \cap \SSy(M_\mathcal{X}) = \emptyset$ for all $\mathcal{X} \subseteq 2^\omega$.
\end{corollary}
 
\begin{proof}
It is easy to see that the poset $P$ in the proof of Theorem \ref{thm:nontrivial-continuum-SSy} satisfies the countable chain condition,
 even if $N$ is uncountable.
By $\ZFC + \MA$, we can apply Lemmata \ref{lem:nontrivial-SSy-extension}, \ref{lem:nontrivial-SSy-incomparability} and \ref{lem:nontrivial-SSy-avoidance}
 to $N$ and $\mathcal{A}$ both of cardinality less than the continuum,
 and obtain a filter $\mathcal{F} \subset P$ s.t.
\begin{enumerate}
 \item Each $p \in \mathcal{F}$ has an extension $q \in \mathcal{F}$ with $n_q > n_p$;
 \item For each $p \in \mathcal{F}$ and each $N$-definable function $F: N^k \to N$,
  there exist $q \in \mathcal{F}$ and $m < \omega$, s.t.
   $q \leq p$, and
   if $\sigma_1,\ldots,\sigma_k \in 2^{n_{q}}$ and $\sigma \in 2^{n_{q}} - \{\sigma_1,\ldots,\sigma_k\}$
   then $q \vdash \exists n < m ((x_\sigma)_n \neq (F(x_{\sigma_1},\ldots,x_{\sigma_k}))_n)$;
 \item For each $p \in \mathcal{F}$, each $g \in \mathcal{A}$ and each $N$-definable function $F: N^k \to N$,
  there exist $q \in \mathcal{F}$ and $m < \omega$ s.t.
   $q \leq p$,
   every $\sigma_1,\ldots,\sigma_k \in (2^{n_{q}})^k$ corresponds to some $n < m$ with
    $q \vdash g(n) \neq (F(x_{\sigma_1}, \ldots, x_{\sigma_k}))_n$.
\end{enumerate}
Then we define $G_\mathcal{F}$ to be the set of formulas in $(x_f: f \in 2^\omega)$
 s.t. every formula in $G_\mathcal{F}$ has a reduct in some $p \in \mathcal{F}$.
It can be proved that $G_\mathcal{F}$ is a type of $N$,
 similar to Lemma \ref{lem:nontrivial-SSy-generic-type}.
Finally, take a realization $(a_f: f \in 2^\omega)$ of $G_\mathcal{F}$
 and let $M_\mathcal{X}$ be an extension of $N$ generated by $N \cup \{a_f: f \in \mathcal{X}\}$.
\end{proof}

Corollary \ref{cor:MA-nontrivial-continuum-SSy} can be extended to
 a partial answer to the Scott Set Problem.

\begin{corollary}[$\ZFC + \MA$]\label{cor:MA-SSy-cone-avoidance}
Suppose that $M$ is a countable non-standard model of $\PA$,
 $\mathcal{A} \subset 2^\omega$ is of cardinality $< 2^\omega$,
 and $\mathcal{B} \subset 2^\omega$ is countable
 and s.t. the Turing ideal generated by $\SSy(M) \cup \mathcal{B}$ is disjoint from $\mathcal{A}$.
Then there exists a family $(M_\mathcal{X}: \mathcal{X} \subseteq 2^\omega)$ s.t.
 $M \prec M_\mathcal{X}$, $\mathcal{B} \subseteq \SSy(M_\mathcal{X})$, $|M_\mathcal{X}| = |\SSy(M_\mathcal{X})| = \max\{\omega, |\mathcal{X}|\}$,
 and
$$
  \mathcal{X} \subseteq \mathcal{Y} \Leftrightarrow M_\mathcal{X} \preceq M_\mathcal{Y} 
  \Leftrightarrow \SSy(M_\mathcal{X}) \subseteq \SSy(M_\mathcal{Y}).
$$
\end{corollary}

\begin{proof}
By $\MA$ and well-known recursion theoretic technique (e.g., see \cite[Lemma 2.6]{Jockusch.Soare:1972.TAMS}),
 we can construct a countable Scott set $\mathcal{S}$ s.t.
 $\SSy(M) \cup \mathcal{B} \subseteq \mathcal{S}$
 and $\mathcal{A} \cap \mathcal{S} = \emptyset$.

By Ehrenfeucht's Theorem \ref{thm:Ehrenfeucht},
 $M$ has an elementary extension $N$ with $\SSy(N) = \mathcal{S}$.
The conclusion then follows from an application of Corollary \ref{cor:MA-nontrivial-continuum-SSy}
 to $N$ and $\mathcal{A}$.
\end{proof}

By Corollary \ref{cor:MA-SSy-cone-avoidance},
 for a Scott set $\mathcal{S}$ which is possibly of cardinality the continuum,
 if we pick $\mathcal{A} \subseteq 2^\omega - \mathcal{S}$ of cardinality less than the continuum
 and also a countable $\mathcal{B} \subseteq \mathcal{S}$,
 then we can find a non-standard $M \models \PA$
  s.t. $|\SSy(M)| = 2^\omega$,
   $\mathcal{A} \cap \SSy(M) = \emptyset$ and $\mathcal{B} \subset \SSy(M)$.
So Corollary \ref{cor:MA-SSy-cone-avoidance} can be regarded as a partial answer to the Scott Set Problem.

\end{document}